\documentclass[preprint]{elsarticle}

\usepackage{graphicx,subfigure}
\usepackage{amsmath}
\usepackage{amsfonts}
\usepackage{amsthm}
\usepackage{latexsym,amssymb}
\usepackage{multicol}

\newtheorem{thm}{Theorem}

\newtheorem{lem}{Lemma}

\newdefinition{rmk}{Remark}
\newdefinition{definition}{Definition}
\newdefinition{obs}{Observation}

\begin{document}

\title{Quasi-efficient domination in grids}

\author[math]{S.A. Aleid\fnref{fn1}}
\ead{sahar.aleid@ual.es}
\author[math]{J. C\'aceres\fnref{fn2}}
\ead{jcaceres@ual.es}
\author[math]{M.L. Puertas\corref{cor1}\fnref{fn2}}
\ead{mpuertas@ual.es}

\cortext[cor1]{Corresponding author}

\fntext[fn1]{Supported by Phoenix Project (Erasmus Mundus Programme).}
\fntext[fn2]{Partially supported by Junta de Andaluc\'ia FQM305 and MINECO MTM2014-60127-P.}

\address[math]{Department of Mathematics, University of Almer\'ia.}

\begin{abstract}
Domination of grids has been proved to be a demanding task and with the addition of independence it becomes more challenging. It is known that no grid with $m,n \geq 5$ has an efficient dominating set, also called perfect code, that is, an independent vertex set such that each vertex not in it has exactly one neighbor in that set. So it is interesting to study the existence of independent dominating sets for grids that allow at most two neighbors, such sets are called independent $[1,2]$-sets. In this paper we prove that every grid has an independent $[1,2]$-set, and we develop a dynamic programming algorithm using min-plus algebra that computes $\imath\dot{}_{[1,2]}(P_m\Box P_n)$, the minimum cardinality of an independent $[1,2]$-set for the grid graph $P_m\square P_n$. We calculate $\imath\dot{}_{[1,2]}(P_m\Box P_n)$ for $2\leq m\leq 13, n\geq m$ using this algorithm, meanwhile the parameter for grids with $14\leq m\leq n$ is obtained through a quasi-regular pattern that, in addition, provides an independent $[1,2]$-set of minimum size.

\end{abstract}


\maketitle

\section{Introduction}

Perfect codes have played a central role in the development of error-correcting codes theory. A \emph{code} in a graph is a vertex set such that any two vertices in it are at distance at least $3$. If, in addition, every vertex not in the code has a neighbor in it, the code is called \emph{perfect} \cite{Big}. A set $S$ of vertices in a graph $G$ is called \emph{independent} if no two vertices in $S$ are adjacent and it is called \emph{dominating} if every vertex not in $S$ has at least one neighbor in $S$. Therefore a perfect code is an independent dominating set such that every vertex not in the set has a unique neighbor in it. These sets are also called \emph{efficient dominating} sets \cite{Ban}, in the sense that they represent the situation where each vertex on the graph is minimally dominated.

It is well known that all efficient dominating sets in a graph $G$ have the same cardinality  $\gamma(G)$ the \emph{domination number}, that is, the cardinal of a minimum dominating set in $G$. Unfortunately the existence of this type of dominating sets is not guaranteed in every graph and it has been extensively studied. For instance the path $P_m$ has an efficient dominating set with $\gamma(P_m)=\lceil \frac{m}{3}\rceil$ vertices for each $m\geq 1$, however the grid graph, which is the cartesian product $P_m\Box P_n$ of two paths, has no efficient dominating set unless $m=n=4$ or $m=2,n=2k+1$  $(2\leq m\leq n)$ \cite{Mar}. In these cases a less demanding construction could be keeping domination and independence but admitting more neighbors, for vertices not in the dominating set. This idea leads to the definition of \emph{independent $[1,k]$-set} which is an independent dominating set $S$ of a graph $G$ such that every vertex $v\in V(G)\setminus S$ has at least one and at most $k$ neighbors in $S$  \cite{Che}. Clearly independent $[1,1]$-sets are precisely efficient dominating sets and independent $[1,2]$-sets could be considered as quasi-efficient dominating sets.

The calculation of domination parameters in grids has proved to be a difficult task. Indeed finding $\gamma(P_m\Box P_n)$ was an open problem for almost 30 years, since it was first studied in \cite{Jac} in relation with  Vizing's Conjecture \cite{Viz} which is still open. An important
milestone on the way to the solution is the upper bound $\gamma(P_m\Box P_n)\leq \big\lfloor\frac{(m+2)(n+2)}{5}\big\rfloor -4$, for  $8\leq m\leq n$  \cite{Cha}. In the same work the author also conjectured that equality is achieved in case $16\leq m\leq n$. The problem was completely solved in \cite{Gon} as authors were able to adapt the ideas in \cite{Gui} to confirm the conjecture. Meantime different efforts were made to calculate exact values of $\gamma(P_m\Box P_n)$, for  small fixed $m$ and for every $n\geq m$.

Among the different techniques used to address this problem, we would like to focus on a dynamic programming algorithm  by applying the (min,+) matrix multiplication. Values of $\gamma(P_m\Box P_n)$ for $m\leq 19$ and $n\geq m$ were obtained with this algorithm in \cite{Spa}. Similar ideas have been recently used to calculate the \emph{independent domination number} of grids \cite{CreOs}, which is the minimum cardinality of an independent dominating set for the graph $P_m\Box P_n$.

In this paper we adapt the constructions in \cite{CreOs,Spa} to solve the open problem proposed in \cite{Che} about the existence of independent $[1,2]$-sets in grids and we also compute the \emph{independent $[1,2]$-number} $\imath\dot{}_{[1,2]}(P_m\square P_n)$ which is the minimum cardinality of such a set, in every grid. To this end, in Section~\ref{sec:small} we devise a dynamic programming algorithm to compute $\imath\dot{}_{[1,2]}(P_m\square P_n)$. This algorithm could be theoretically applied in every grid, however the long running time needed makes it useful just on grids of small size, in our case $m\leq 13$ and $n\geq m$. On the other hand in Section~\ref{sec:big} we compute $\imath\dot{}_{[1,2]}(P_m\Box P_n)$ in grids with $14\leq m\leq n$ by means of a quasi-regular pattern, resembling the mentioned above construction for the upper bound of the domination number \cite{Cha}. In addition this pattern explicitly provides an independent $[1,2]$-set of minimum size.

All the graphs considered here are finite, undirected, simple and connected. For undefined basic concepts we refer the reader to basic graph theoretical literature as~\cite{chlepi11, Hay}.

\section{A dynamic programming algorithm to calculate $\imath\dot{}_{[1,2]}(P_m\Box P_n)$}\label{sec:small}

In this section we present a dynamic programming algorithm to prove the existence of independent $[1,2]$-sets in the grid $P_m\square P_n$ and to obtain $\imath\dot{}_{[1,2]}(P_m\square P_n)$, following the ideas in \cite{CreOs,Spa}.

Throughout this paper we will always consider the grid graph $P_m\square P_n$ as an array with $m$ rows, $n$ columns and with vertex set $\{ v_{ij}\colon 1\leq i\leq m,\ 1\leq j\leq n\}$. Let $S$ be an independent $[1,2]$-set of $P_m\square P_n$.
We define a labeling of vertices of $P_m\square P_n$ associated to $S$ as follows

$$
l(v_{ij})=
\left\{
\begin{array}{ll}
0&\text{ if }v_{ij}\in S\\
1&\text{ if }v_{ij}\notin S\text{ and }\vert \{v_{i (j-1)}, v_{(i-1) j}, v_{(i+1) j}\}\cap S\vert =1 \\
2&\text{ if }v_{ij}\notin S\text{ and }\vert \{v_{i (j-1)}, v_{(i-1) j}, v_{(i+1) j}\}\cap S\vert =2 \\
3&\text{ if }v_{ij}\notin S\text{ and }\vert \{v_{i (j-1)}, v_{(i-1) j}, v_{(i+1) j}\}\cap S\vert =0 \\
\end{array}
\right.
$$

Each label shows whether a vertex is in $S$ or not, and in the second case it shows how many neighbors does that vertex have in $S$, these neighbors are either belong to the vertex column or the previous one. If these cases do not hold, the label shows that the vertex has a unique neighbor in the following column.
Given an independent $[1,2]$-set, we can identify each vertex with its label so we obtain an array of labels with $m$ rows and $n$ columns. Hereinafter, given an independent $[1,2]$-set of $P_m\square P_n$, the columns of the grid are words of length $m$ in the alphabet $\{0,1,2,3\}$ and the number of zeros in the array is the cardinality of the independent $[1,2]$-set. The algorithm considers all the arrays of words (as columns) that come from some independent $[1,2]$-set of $P_m\square P_n$ and it calculates the minimum among the number of $0's$ in each array. This minimum is equal to $\imath\dot{}_{[1,2]}(P_m\square P_n)$.

It is clear that not every word of length $m$ can belong to such labeling, for instance there can not be two consecutive $0's$ in a word, because of independence. The first objective is to identify the words that belong to the labeling associated to some independent $[1,2]$-set.

We need to have in mind that not every word can be in the first column nor in the last one. A word in the first column has no neighbors in the previous one and a word in the last column has no label $3$, because they would not be dominated. The second task is to identify those words that can be in the first column and that can be in the last one.

It is also clear that not any two words can follow each other in a labeling associated to some independent $[1,2]$-set, for instance if a word has $3$ in the $r^{th}$ position, then the following word must have $0$ in the $r^{th}$ position, to preserve domination. The last objective regarding to words is to identify which of them can follow a given one.

The final part of the algorithm calculates the minimum number of zeros in an array, among all possible arrays of labels associated to an independent $[1,2]$-set, by means of the successive addition of columns.

For the rest of this section, let $S$ be an independent $[1,2]$-set of $P_m\square P_n$ and identify the graph with the array of the associated labels.

\subsection{The suitable words}

We are going to identify the words that belong to the labeling associated to some independent $[1,2]$-set. It is clear that a column can not contain two consecutive $0's$ because of independence, nor two consecutive $2's$ because in the previous column should be two consecutive $0's$, also it can not contain two consecutive $3's$ because in the following column will be two consecutive $0's$.

Moreover a column can not contain any of the sequences $03, 30, 010$ by the definition of labeling. If a column contains the sequence $11$ then the previous label or the following label in such column (or both) must be $0$, because in other case would be two consecutive $0's$ in the previous column. If a column contains the sequence $32$ then the other label next to $2$ in the column must be $0$, by the definition of  labeling. Similarly sequence $23$ must be preceded by $0$ and both sequences $21, 12$ must be placed between two $0's$.

\begin{definition}
A word of length $m$ in the alphabet $\{0,1,2,3\}$ satisfying all the rules described above is called $\emph{suitable}$.
\end{definition}

We denote the cardinal of the set of all suitable words by $k$, so every suitable word can be identified with $p\in \{1,\dots ,k\}$.

\subsection{The first column, the last column and the initial vector}

The first column is a suitable word with no neighbors in the previous column, as such column does not exist. Similarly the last column does not contain any $3$, because in this case these vertices would be not dominated. With these ideas we pose the following definition.

\begin{definition}

A suitable word is called \emph{initial} if every vertex labeled as $2$ is placed between two $0's$ and every vertex labeled as $1$ is preceded or followed (but not both) by $0$
and it is called \emph{final} if it has no vertex labeled as $3$.

The \emph{initial vector} $X^1$ is a column vector of size $k$ such that for every suitable word $p\in \{1,2,\dots , k\}$, the $p^{th}$ entry of the vector is
$$X^1(p)=
\left\{
\begin{array}{lcl}
\text{number of zeros of word }p &\dots & \text{ if } p \text{ is an initial word}\\
\infty &\dots & \text{ if } p \text{ is not an initial word}
\end{array}
\right.
$$
\end{definition}

\subsection{The rules for adding a column and the transition matrix}

Not every pair of suitable words can be consecutive columns in the labeling associated to some independent $[1,2]$-set, in order to preserve both independence and $[1,2]$-domination. Next we show the conditions needed to ensure that the word $p=p_1\dots p_m$ can follow the word $q=q_1\dots q_m$.
\par\medskip

If $q_i=0$, then
$\left\{
\begin{array}{l}
(\text{\it if } i=1)
\left\{
\begin{array}{l}
 p_1=1, p_2\neq 0 \ or\\
 p_1=2, p_2=0
 \end{array}
 \right. \\
 (\text{\it if }1<i<m)
 \left\{
 \begin{array}{l}
  p_i=1, p_{i-1}\neq 0, p_{i+1}\neq0 \ or \\
  p_i=2, p_{i-1}=0 \ or\ p_{i+1}=0\\

  \end{array}
  \right. \\
  (\text{\it if } i=m)
  \left\{
  \begin{array}{l}
  p_m=1, p_{m-1}\neq 0 \ or \\
  p_m=2, p_{m-1}=0
   \end{array}
   \right.
   \end{array}
 \right.$

 \par\bigskip

If $q_i=1$, then
$\left\{
\begin{array}{l}
(\text{\it if } i=1)
\left\{
\begin{array}{l}
p_1=0 \ or \\
p_1=1, p_{2}=0\ or \\
p_1=3
\end{array}
\right. \\
\\
(\text{\it if }1<i<m)
\left\{
\begin{array}{l}
p_i=0 \ or \\
p_i=1, p_{i-1}=0, p_{i+1}\neq 0  \ or \\
p_i=1, p_{i-1}\neq 0, p_{i+1}=0 \ or \\
p_i=2, p_{i-1}=0, p_{i+1}=0\ or \\
p_i=3
\end{array}
\right. \\
\\
(\text{\it if } i=m)
\left\{
\begin{array}{l}
p_m=0 \ or \\
p_m=1, p_{m-1}=0\ or \\
p_m=3
\end{array}
\right.
\end{array}
\right.
$

\par\bigskip

If $q_i=2$, then
$\left\{
\begin{array}{l}
(\text{\it if } i=1) \ p_1=3 \\
\\
(\text{\it if }1<i<m)
\left\{
\begin{array}{l}
p_i=1, p_{i-1}=0, p_{i+1}\neq 0  \ or \\
p_i=1, p_{i-1}\neq 0, p_{i+1}=0 \ or \\
p_i=3
\end{array}
\right.\\
\\
(\text{\it if } i=m) \ p_m=3
\end{array}
\right.
$

\par\bigskip

If $q_i=3$, then $p_i=0$, for $1\leq i\leq m$.

\par\bigskip

These properties lead us to the following definition.

\begin{definition}
We say that \emph{word $p$ can follow word $q$} if they satisfy all the properties described above.

The \emph{transition matrix} is the square matrix $A$ of size $k$ such that, for every pair $p, q\in \{1,2,\dots , k\}$, the entry $A_{pq}$ in row $p$ and column $q$ is defined as follows:
$$A_{pq}=
\left\{
\begin{array}{lcl}
\text{number of zeros of word }p &\dots & \text{ if word } p \text{ can follow word }q \\
&&\\
\infty &\dots &\text{ if word } p \text{ can not follow word }q
\end{array}
\right.
$$
\end{definition}

\subsection{Adding many columns via the min-plus multiplication}

Recall from \cite{CreOs} that $(min,+)$-algebra is a semiring which is defined by two operators on the real numbers, including infinity, $\boxplus$ representing minimization and $\boxtimes$ representing standard addition. The $(min,+)$-algebra, also called tropical \cite{Pin}, can be defined on matrices as $(A\boxtimes B)_{i,j}=\ \boxplus_{k=1}^{n}(a_{i,k}\boxtimes b_{k,j})$, where $A=(a_{i,j})$ and $B=(b_{i,j})$ are respectively, $m\times n$ and $n\times p$ matrices. Using this matrix multiplication, we can obtain, from the initial vector $X^1$ and the transition matrix $A$, vectors $X^2=A\boxtimes X^{1}, X^3=A\boxtimes X^{2}, \dots ,X^n=A\boxtimes X^{(n-1)}$. These vectors allows us to characterize grids having an independent $[1,2]$-set. To this end we first need the following lemma.

\begin{lem}\label{lem:finite}
Let $r\geq 2$ be an integer and let $p_r$ be a suitable word of length $m$.
\begin{enumerate}
\item $X^r(p_r)<\infty$ if and only if there exists suitable words $p_1,p_2,\dots p_{r-1}$ such that $p_1$ is initial and $p_i$ can follow $p_{i-1}$ for $2\leq i\leq r$.
\item If $X^n(p_n)<\infty$ then $X^n(p_n)$ is the minimum number of zeros in a labeling of vertices of $P_m\square P_n$, which has word $p_n$ in the $n^{th}$-column and such that the set of vertices labeled as zero is independent and $[1,2]$-dominates the graph, except vertices in the $n^{th}$-column with label $3$.
\end{enumerate}
\end{lem}

\begin{proof}
\begin{enumerate}
\item Suppose that $r=2$, then $X^2(p_2)=\min \{ A_{p_2 1}+X^1(1), A_{p_2 2}+X^1(2),\dots , A_{p_2 k}+X^1(k)\}$, by definition of the $(min,+)$ matrix multiplication. Therefore $X^2(p_2)<\infty$ if and only if there exists an initial word $p_1$ such that $A_{p_2 p_1}<\infty$ or equivalently that $p_2$ can follow $p_1$.

By using induction,
assume that the statement is true of $r-1\geq 2$ and let $p_r$ be any suitable word. Then $X^r(p_r)=\min \{ A_{p_r 1}+X^{r-1}(1), A_{p_r 2}+X^{r-1}(2),\dots , A_{p_r k}+X^{r-1}(k)\}$, so $X^r(p_r)<\infty$ if and only if there exists a word $p_{r-1}$ such that both conditions $X^{r-1}(p_{r-1})<\infty$ and $A_{p_r p_{r-1}}<\infty$ hold. By the inductive hypothesis the first condition is equivalent to the existence of suitable words $p_1,p_2,\dots p_{r-2}$ such that $p_1$ is initial and $p_i$ can follow $p_{i-1}$ for $2\leq i\leq r-1$ and the second condition is equivalent to $p_r$ can follow $p_{r-1}$.

\item Assume that $X^2(p_2)<\infty$ then there exists an initial word $p_1$, so $X^1(p_1)<\infty$, such that $p_2$ can follow $p_1$. Both conditions ensure that, in the array with $p_1$ as first column and $p_2$ as second one, the set of vertices with label zero is independent and $[1,2]$-dominates $P_m\Box P_2$, except for vertices in the second column with label equal to $3$. Moreover $A_{p_2 p_1}+X^1(p_1)$ is equal to the number of vertices labeled as zero in the array, so $X^2(p_2)=\min \{ A_{p_2 1}+X^1(1),\dots , A_{p_2 k}+X^1(k)\}$ is the minimum number of vertices with label zero in an array with $p_2$ as second column.

    Again we will proceed by induction, so assume now that the statement is true for $n-1\geq 2$ and let $p_n$ be a suitable word with $X^n(p_n)<\infty$. Then, there exists a suitable word $p_{n-1}$ such that
    $A_{p_n p_{n-1}}+X^{n-1}(p_{n-1})<\infty$, so $X^{n-1}(p_{n-1})<\infty$ and $p_n$ can follow $p_{n-1}$. By the inductive hypothesis the value of $A_{p_n p_{n-1}}+X^{n-1}(p_{n-1})$ is the number of zeros in a labeling of vertices of $P_m\square P_n$, with the conditions in the statement, which has word $p_n$ in the $n^{th}$-column and $p_{n-1}$ in the $(n-1)^{th}$-column, and such that the number of zeros in the subgraph $P_m\square P_{n-1}$ is minimum. So $X^n(p_n)=\min \{ A_{p_n 1}+X^{n-1}(1),\dots , A_{p_n k}+X^{n-1}(k)\}$ is the minimum number of zeros among all the possible arrays which has word $p_n$ in the $n^{th}$-column and with the desired conditions.
\end{enumerate}
\end{proof}

Now we can characterize grids having an independent $[1,2]$-set, using the vectors obtained by means of the $(\min, +)$ matrix multiplication.

\begin{thm}\label{thm:existence}
There exists an independent $[1,2]$-set in $P_m\square P_r$ if an only if there exists a final word $p$ such that $X^r(p)<\infty$. Moreover in this case
$$\imath\dot{}_{[1,2]}(P_m\square P_r)=\min \{ X^r(p)\colon p \text{ is a final word} \}$$
\end{thm}

\begin{proof}
 By condition 1 of Lemma~\ref{lem:finite}, $X^r(p_r)<\infty$ for a final word $p_r$ is equivalent to the existence of suitable words $p_1, p_2,\dots p_{r-1}, p_r$ such that $p_1$ is initial, $p_i$ can follow $p_{i-1}$ for $2\leq i\leq r$ and $p_r$ is final. Clearly  this condition is equivalent to the existence of an independent $[1,2]$-set in $P_m\square P_r$.

Finally by condition 2 of Lemma~\ref{lem:finite}, $\min \{ X^r(p)\colon p \text{ is a final word} \}$ is the minimum number of zeros in a labeling of vertices of $P_m\square P_r$, such that the set of vertices labeled as zero is independent and $[1,2]$-dominates the graph, so $\imath\dot{}_{[1,2]}(P_m\square P_r)=\min \{ X^r(p)\colon p \text{ is a final word} \}$.
\end{proof}

\subsection{The recursion rule to calculate $\imath\dot{}_{[1,2]}(P_m\square P_n)$}

The above theorem allows us to know if a fixed grid $P_m\square P_r$ has an independent $[1,2]$-set. We now provide a recurrence argument that ensures the existence of such sets in $P_m\square P_n$ for every $n$ big enough and that gives a formula for the independent $[1,2]$-number in those cases. We use the following result which is similar to Theorem 2.2 in \cite{CreOs}.

\begin{thm}~\label{thm:recurrence}
 Suppose that there exist integers $n_0, c, d > 0$ satisfying the equation $X^{n_{0}+d}{(p)}=X^{n_0}{(p)}+c$, for every suitable word $p$. If $P_m\square P_{r}$ has an independent $[1,2]$-set for $n_0\leq r\leq n_0+d-1$, then
 \begin{enumerate}
 \item $P_m\square P_{n}$ has an independent $[1,2]$-set, for all $n\geq n_0$,
 \item $\imath\dot{}_{[1,2]}(P_m\square P_{n+d})=\imath\dot{}_{[1,2]}(P_m\square P_n)+c$, for all $\ n\geq n_0.$
\end{enumerate}
\end{thm}

\begin{proof}

Since $X^{n_{0}+d}{(p)}-X^{n_0}{(p)}=c$, for all suitable words $p$, then we have  $A^{d}\boxtimes X^{n_0}=c\boxtimes X^{n_0}$. Now using induction we obtain
$X^{n+d}=A\boxtimes X^{n+d-1}=A\boxtimes c\boxtimes X^{n-1}=c\boxtimes X^{n},\ \text{for } n\geq n_0$ or equivalently $X^{n+d}(p)=X^n(p)+c$ for every word suitable word $p$ and every $n\geq n_0$.

 \begin{enumerate}
 \item Let $n\geq n_0$ be a integer. If $n_0\leq n\leq n_0+d-1$, by hypothesis $P_m\square P_{n}$ has an independent $[1,2]$-set so assume that $n_0+d\leq n$. By Theorem~\ref{thm:existence} we just need to show that $\min \{ X^n(p)\colon p \text{ is a final word} \}<\infty$. If $n=n_0+d$, then $\min \{ X^{n_0+d}(p)\colon p \text{ is a final word}\}=\min \{ X^{n_0}(p)+c \colon p \text{ is a final word} \}=\min \{ X^{n_0}(p) \colon p \text{ is a final word} \}+c<\infty$. We now proceed  by induction over $n$, so assume that $n_0+d<n$ and that $\min \{ X^r(p)\colon p \text{ is a final word} \}<\infty$ for $n_0+d\leq r<n$. Then $\min \{ X^{n}(p)\colon $ $p \text{ is a final word}\}=\min \{ X^{n-d}(p)+c \colon p \text{ is a final word} \}=\min \{ X^{n-d}(p) \colon$ $ p \text{ is a final word} \}+c<\infty$.

\item Let $n\geq n_0$, then $n+d\geq n_0$ and both grids $P_m\square P_n$ and $P_m\square P_{n+d}$ have independent $[1,2]$-sets. Moreover

\begin{align*}
\imath\dot{}_{[1,2]}(P_m\square P_{n+d})&=\min \{ X^{n+d}(p)\colon p \text{ is a final word} \}\\
&=\min \{ X^n(p)+c \colon p \text{ is a final word} \}\\
&=\min \{ X^n(p) \colon p \text{ is a final word} \}+c\\
&=\imath\dot{}_{[1,2]}(P_m\square P_n)+c
\end{align*}
\end{enumerate}
\end{proof}

\subsection{Results}
Using the algorithm described in this section we have proved the existence of independent $[1,2]$-sets and we have obtained $\imath\dot{}_{[1,2]}(P_m\square P_n)$ for $2\leq m \leq 13$ and for every $n\geq m$. Calculations took about 18 CPU hours (using a 3.3-GHz Intel Core I3-2120 CPU).

We listed all the suitable words of length $m$ and we constructed the initial vector $X^1$ and the transition matrix $A$. Then we calculated vectors $X^i=A\boxtimes X^{i-1}$, until getting integers $n_0,c,d\geq 1$ such that $X^{n_{0}+d}= c \boxtimes X^{n_{0}}$, where $n_0$ is the smallest integer satisfying the equation. These values can be found in Table~\ref{tab:table1}. According to Theorem~\ref{thm:recurrence}, we obtained the finite difference equation $\imath\dot{}_{[1,2]}(P_m\square P_{n+d})-\imath\dot{}_{[1,2]}(P_m\square P_{n})=c,\ \text{for } n\geq n_0$ (see Table~\ref{tab:table1}).

We show in Table~\ref{tab:table2} the boundary conditions for these finite difference equations, which are the values $\imath\dot{}_{[1,2]}(P_m\square P_r)$ for $n_0\leq r\leq n_0+d-1$, that we calculated by applying Theorem~\ref{thm:existence} to each vector $X^r$. So, by Theorem~\ref{thm:recurrence}, we can ensure the existence of at least one independent $[1,2]$-set in each $P_m\square P_n$ with $2\leq m \leq 13$ and for all $n\geq m$.

{\def\arraystretch{1.5}
\begin{table}[h!]
\centering
\caption{Finite difference equations for $\imath\dot{}_{[1,2]}(P_m\square P_n)=f_{m}(n)$}
\label{tab:table1}
$\scriptsize{
\begin{array}{r|r|r|r|l}
m & n_0&d&c &  \text{finite difference equation}\\
\hline
2&4&2&1&f_{2}(n\!+\!2)-f_{2}(n)\!=\!1\\
\hline
3&7&4&3&f_{3}(n\!+\!4)-f_{3}(n)=3\\
\hline
4&11&1&1&f_{4}(n\!+\!1)-f_{4}(n)=1\\
\hline
5&15&5&6&f_{5}(n\!+\!5)-f_{5}(n)=6\\
\hline
6&9&7&10&f_{6}(n\!+\!7)-f_{6}(n)=10\\
\hline
7&12&3&5&f_{7}(n\!+\!3)-f_{7}(n)=5\\
\hline
8&18&8&15&f_{8}(n\!+\!8)-f_{8}(n)=15\\
\hline
9&28&10&21&f_{9}(n\!+\!10)-f_{9}(n)=21\\
\hline
10&46&9&21&f_{10}(n\!+\!9)-f_{10}(n)=21\\
\hline
11&50&11&28&f_{11}(n\!+\!11)-f_{11}(n)=28\\
\hline
12&27&13&36&f_{12}(n\!+\!13)-f_{12}(n)=36\\
\hline
13&73&12&3&f_{13}(n\!+\!12)-f_{13}(n)=3\\
\hline
\end{array}
%}
}$
\end{table}
}

\newpage

{\def\arraystretch{1.5}
\begin{table}[h!]
\centering
\caption{Boundary conditions}
\label{tab:table2}
$%\scriptsize{
\begin{array}{l|l|l|l}
m & n_0&d&\text{boundary conditions: } \imath\dot{}_{[1,2]}(P_m\square P_r), n_0\leq r\leq n_0+d-1\\
\hline
2&4&2&f_{2}(4)\!=\!3,f_{2}(5)\!=3\!\\
\hline
3&7&4&f_{3}(7)\!=\!6,f_{3}(8)\!=\!7,f_{3}(9)\!=\!7,f_{3}(10)\!=9\!\\
\hline
4&11&1&f_{4}(11)\!=\!11\\
\hline
5&15&5&f_{5}(15)\!=\!19,f_{5}(16)\!=\!20,f_{5}(17)\!=\!22,f_{5}(18)\!=\!23,f_{5}(19)\!=\!24\\
\hline
6&9&7&f_{6}(9)\!=\!14,f_{6}(10)\!=\!16,f_{6}(11)\!=\!17,f_{6}(12)\!=\!18,f_{6}(13)\!=\!20\\&&&f_{6}(14)\!=\!22,f_{6}(15)\!=\!22\\
\hline
7&12&3&f_{7}(12)\!=\!21,f_{7}(13)\!=\!22,f_{7}(14)\!=\!24\\
\hline
8&18&8&f_{8}(18)\!=\!35,f_{8}(19)\!=\!37,f_{8}(20)\!=\!39,f_{8}(21)\!=\!41\\
&&&f_{8}(22)\!=\!43,f_{8}(23)\!=\!45,f_{8}(24)\!=\!47,f_{8}(25)\!=\!48\\
\hline
9&28&10&f_{9}(28)\!=\!60,f_{9}(29)\!=\!63,f_{9}(30)\!=\!65, f_{9}(31)\!=\!66\\&&&f_{9}(32)\!=\!69,f_{9}(33)\!=\!71, f_{9}(34)\!=\!73\\&&& f_{9}(35)\!=\!75,f_{9}(36)\!=\!77,f_{9}(37)\!=\!80\\
\hline
10&46&9&f_{10}(46)\!=\!108,f_{10}(47)\!=\!111,f_{10}(48)\!=\!113\\&&&f_{10}(49)\!=\!115,f_{10}(50)\!=\!118,
f_{10}(51)\!=\!120\\&&&f_{10}(52)\!=\!122,f_{10}(53)\!=\!125,f_{10}(54)\!=\!127\\
\hline
11&50&11&f_{11}(50)\!=\!129,f_{11}(51)\!=\!132,f_{11}(52)\!=\!134,f_{11}(53)\!=\!137\\&&&f_{11}(54)\!=\!139,f_{11}(55)\!=\!142,f_{11}(56)\!=\!144,f_{11}(57)\!=\!147\\&&&
f_{11}(58)\!=\!150,f_{11}(59)\!=\!152,f_{11}(60)\!=\!155\\
\hline
12&27&13&f_{12}(27)\!=\!76,f_{12}(28)\!=\!79,f_{12}(29)\!=\!82,f_{12}(30)\!=\!85\\&&&f_{12}(31)\!=\!88,
f_{12}(32)\!=\!90,f_{12}(33)\!=\!93\\&&&f_{12}(34)\!=\!96,f_{12}(35)\!=\!99 f_{12}(36)=102\\&&&f_{12}(37)=104,f_{12}(38)=107,f_{12}(39)\!=\!110\\
\hline
13&73&12&f_{13}(73)\!=\!220,f_{13}(74)\!=\!224,f_{13}(75)\!=\!227,f_{13}(76)\!=\!229\\&&&f_{13}(77)\!=\!233,
f_{13}(78)\!=\!236,f_{13}(79)\!=\!238,f_{13}(80)\!=\!242\\&&&f_{13}(81)\!=\!245,
f_{13}(82)\!=\!247,f_{13}(83)\!=\!251,f_{13}(84)\!=\!254\\
\hline
\end{array}
%}
$
\end{table}
}

The unique solution obtained from solving the finite difference equation for each $m$ is the desired formula for $\imath\dot{}_{[1,2]}(P_m\square P_n)$, $n\geq n_0$. In addition, we checked the values $\imath\dot{}_{[1,2]}(P_m\square P_s)$ for $m\leq s\leq n_0-1$ to obtain a formula for the independent $[1,2]$-number of grid $P_m\square P_n,$ with $2\leq m\leq 13$ and $n\geq m$.

The complete list is the following.
\par\medskip

$\imath\dot{}_{[1,2]}(P_2\square P_n)=\displaystyle\Big\lfloor \frac{n+2}{2} \displaystyle\Big\rfloor$
\par\medskip

$ \imath\dot{}_{[1,2]}(P_3\square P_n)=
\left\{
\begin{array}{lll}
\displaystyle\Big\lfloor \frac{3n+8}{4} \Big\rfloor &\dots& n\equiv 2\pmod 4\\
\\
\displaystyle\Big\lfloor \frac{3n+4}{4} \Big\rfloor &\dots& \text{otherwise}
\end{array}
\right.
$
\par\medskip

$ \imath\dot{}_{[1,2]}(P_4\square P_n)=
\left\{
\begin{array}{lll}
\displaystyle\ n+1 &\dots& n=5,6,9 \\
\\
\displaystyle\ n&\dots&  \text{otherwise}
\end{array}
\right.
$
\par\medskip

$ \imath\dot{}_{[1,2]}(P_5\square P_n)=\displaystyle\Big\lfloor \frac{6n+8}{5} \Big\rfloor$
\par\medskip

$ \imath\dot{}_{[1,2]}(P_6\square P_n)=
\left\{
\begin{array}{lcl}
\displaystyle\Big\lfloor \frac{10n+17}{7} \Big\rfloor &\dots& n\equiv 0,3\pmod 7,\ n\neq 7 \\
\\
\displaystyle\Big\lfloor \frac{10n+10}{7} \Big\rfloor&\dots& \text{otherwise}
\end{array}
\right.
$
\par\medskip

$\imath\dot{}_{[1,2]}(P_7\square P_n)=\displaystyle\Big\lfloor \frac{5n+3}{3} \Big\rfloor$
\par\medskip

$ \imath\dot{}_{[1,2]}(P_8\square P_n)=
\left\{
\begin{array}{lll}
\displaystyle\ 16 &\dots& n=8 \\
\\
\displaystyle\Big\lfloor \frac{15n+16}{8} \Big\rfloor &\dots& \text{otherwise}
\end{array}
\right.
$
\par\medskip

$ \imath\dot{}_{[1,2]}(P_9\square P_n)=
\left\{
\begin{array}{lll}
\displaystyle\Big\lfloor \frac{21n+28}{10} \Big\rfloor &\dots& n\equiv 0,7,9\pmod {10}\\
\\
\displaystyle\Big\lfloor \frac{21n+18}{10} \Big\rfloor &\dots& \text{otherwise}
\end{array}
\right.
$
\par\medskip

$ \imath\dot{}_{[1,2]}(P_{10}\square P_n)=
\left\{
\begin{array}{lcl}
\displaystyle\Big\lfloor \frac{21n+23}{9} \Big\rfloor &\dots& n=12,18,21,30\\
\\
\displaystyle\Big\lfloor \frac{21n+14}{9} \Big\rfloor &\dots& \text{otherwise}
\end{array}
\right.
$
\par\medskip

$\imath\dot{}_{[1,2]}(P_{11}\square P_n)=\displaystyle\Big\lfloor \frac{28n+26}{11} \Big\rfloor$
\par\medskip

$\imath\dot{}_{[1,2]}(P_{12}\square P_n)=
\left\{
\begin{array}{lcl}
\displaystyle\Big\lfloor \frac{36n+41}{13} \Big\rfloor &\dots& n\equiv 10\pmod {13}\\
\\
\displaystyle\Big\lfloor \frac{36n+28}{13} \Big\rfloor&\dots& \text{otherwise}
\end{array}
\right.
$
\par\medskip

$\imath\dot{}_{[1,2]}(P_{13}\square P_n)=
\left\{
\begin{array}{lcl}
\displaystyle\ 3n+1 &\dots& n\equiv 1,4,7,10\pmod {12}\\
\\
\displaystyle\ 3n+2 &\dots& \text{otherwise}
\end{array}
\right.
$

\section{Independent $[1,2]$-sets and independent $[1,2]$-number in big grids}\label{sec:big}

In this section we follow the construction used in \cite{CreOs} to obtain an independent $[1,2]$-set in grids $P_m\Box P_n$, for $14\leq m\leq n$.
We also calculate the exact value of $\imath\dot{}_{[1,2]}(P_m\Box P_n)$, for $14\leq m\leq n$. In fact we notice that the same proof of Theorem 3.1 of \cite{CreOs} to determine the independent domination number of grids $P_m\Box P_n$, with $16\leq m\leq n$, is also valid for the independent $[1,2]$-number. On the other hand we slightly modify this proof to adapt it to the remaining cases $14=m\leq n$ and $15=m\leq n$.

\begin{thm}
If $14\leq m\leq n$ then
$$\imath\dot{}_{[1,2]}(P_m\Box P_n)=\Bigg \lfloor \frac{(m+2)(n+2)}{5}\Bigg\rfloor -4$$
\end{thm}

\proof
Let $m,n$ be integers such that $14\leq m\leq n$. In \cite{CreOs} it is shown that
\begin{align*}
i(P_{14}\Box P_n)= &\Big \lfloor \frac{16n+12}{5}\Big\rfloor=\Big \lfloor \frac{(14+2)(n+2)}{5}\Big\rfloor -4, \text{ for } 14\leq n\\
i(P_{15}\Box  P_n)=& \Big \lfloor \frac{17n+14}{5}\Big\rfloor=\Big \lfloor \frac{(15+2)(n+2)}{5}\Big\rfloor -4,\text{ for }15\leq n\\
i(P_m\Box  P_n)=&\Big \lfloor \frac{(m+2)(n+2)}{5}\Big\rfloor -4,\text{ for } 16\leq m\leq n
\end{align*}
Using that the independent domination number is a lower bound of the independent $[1,2]$-number, if this exists, we just need to prove the inequality $\imath\dot{}_{[1,2]}(P_m\Box  P_n)\leq \big \lfloor \frac{(m+2)(n+2)}{5}\big\rfloor -4$ and to this end we will construct an independent $[1,2]$-set with this number of vertices for each case, solving also the question of the existence of such sets.

Examples of independent $[1,2]$-sets of $P_{14}\Box P_n$ with $\big\lfloor \frac{(14+2)(n+2)}{5}\big\rfloor - 4$ vertices are shown in Figure~\ref{fig:case_14}, for $n=14,15,16,17$. Similarly in Figure~\ref{fig:case_15} we show examples of independent $[1,2]$-sets of $P_{15}\Box P_n$ with $\big\lfloor \frac{(15+2)(n+2)}{5}\big\rfloor - 4$ vertices, for $n=15,16,17$.

\begin{figure}[htbp]
  \begin{center}
    \subfigure[$\imath\dot{}_{[1,2]}(P_{14}\Box P_{14})\leq 47$]{\includegraphics[width=0.3\textwidth]{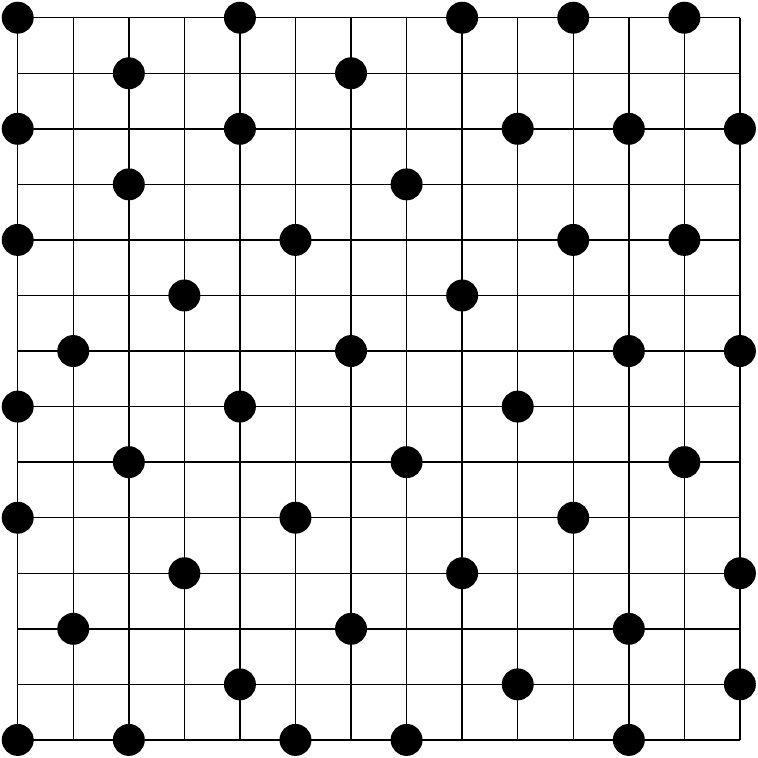}\label{fig:case_14_14}} \hspace{1cm}
    \subfigure[$\imath\dot{}_{[1,2]}(P_{14}\Box P_{15})\leq 50$]{\includegraphics[width=0.32\textwidth]{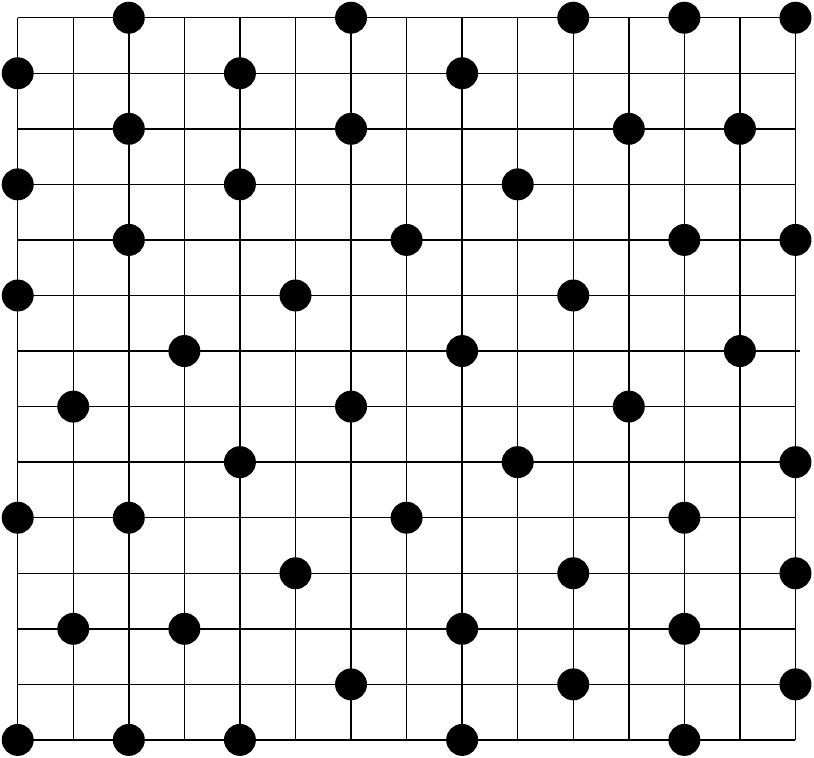}\label{fig:case_14_15}}\hspace{0.7cm}\\

     \subfigure[$\imath\dot{}_{[1,2]}(P_{14}\Box P_{16})\leq 53$]{\includegraphics[width=0.33\textwidth]{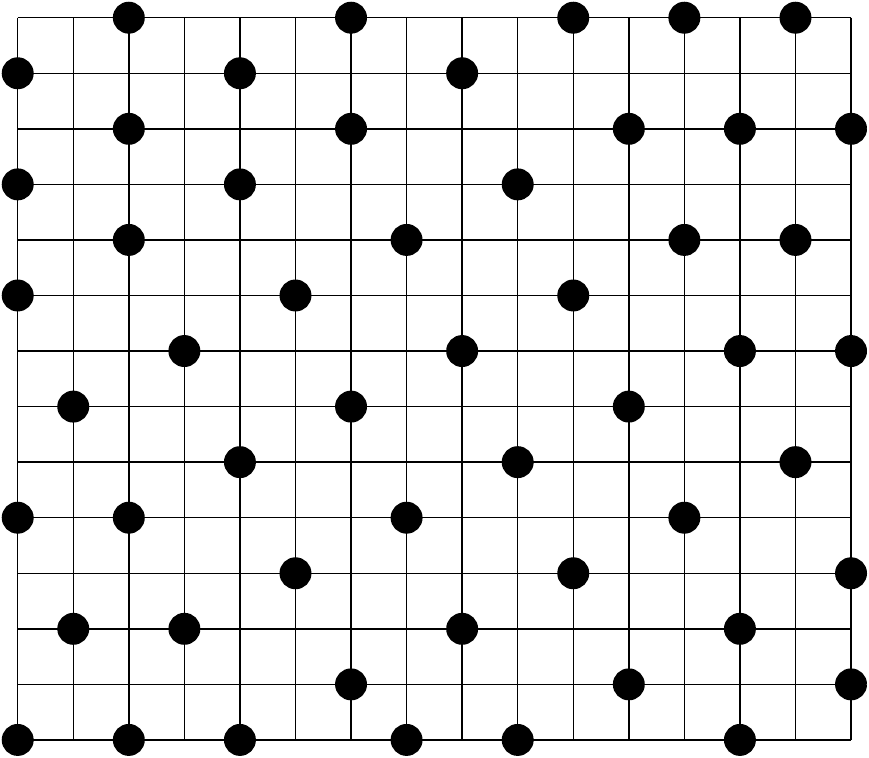}\label{fig:case_14_16}}\hspace{1cm}
      \subfigure[$\imath\dot{}_{[1,2]}(P_{14}\Box P_{17})\leq 56$]{\includegraphics[width=0.35\textwidth]{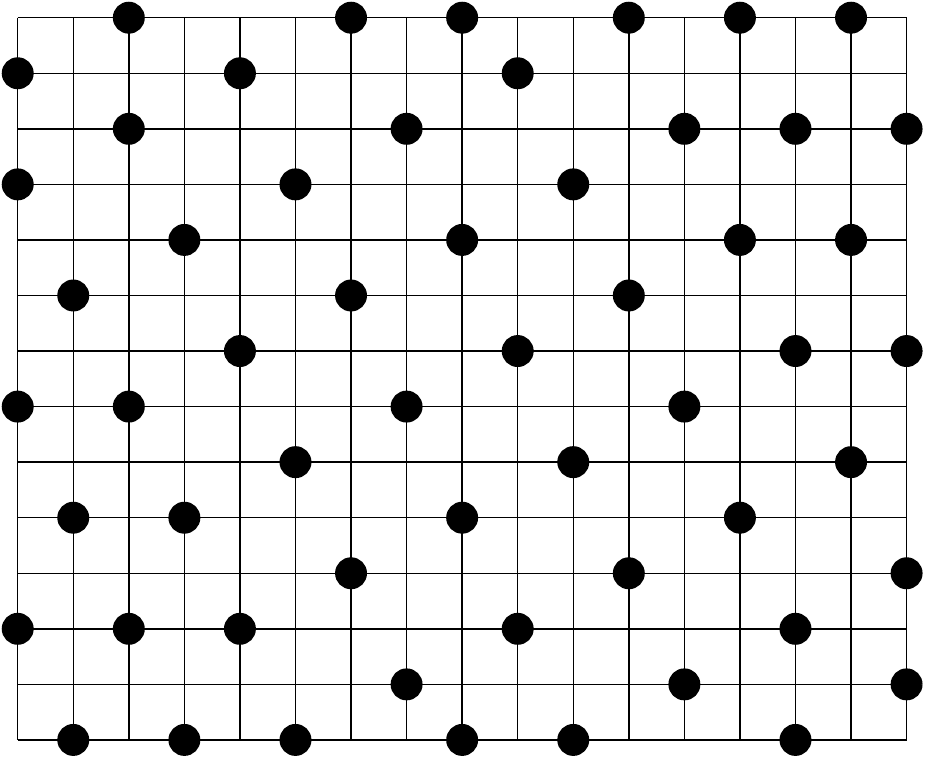}\label{fig:case_14_17}}
\caption{Independent $[1,2]$-sets of $P_{14}\Box P_{n}$ with $\Big\lfloor \frac{(14+2)(n+2)}{5}\Big\rfloor-4$ vertices, for $14\leq n\leq 17$.}\label{fig:case_14}
  \end{center}
\end{figure}

\begin{figure}[htp]
  \begin{center}
    \subfigure[$\imath\dot{}_{[1,2]}(P_{15}\Box P_{15})\leq 53$]{\includegraphics[width=0.27\textwidth]{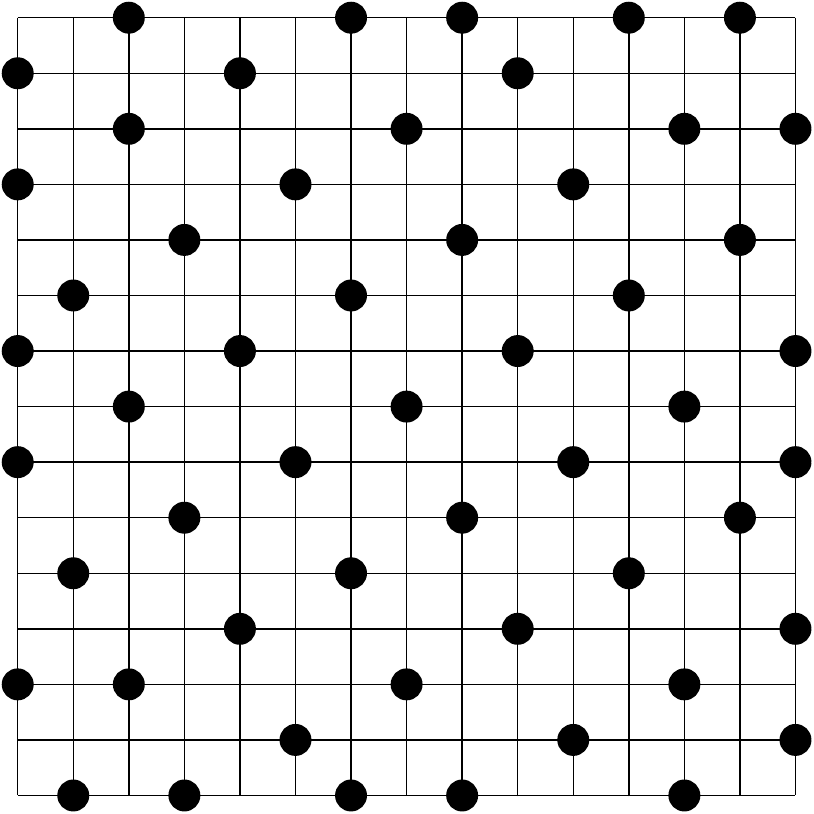}\label{fig:case_15_15}} \hspace{0.4cm}
    \subfigure[$\imath\dot{}_{[1,2]}(P_{15}\Box P_{16})\leq 57$]{\includegraphics[width=0.30\textwidth]{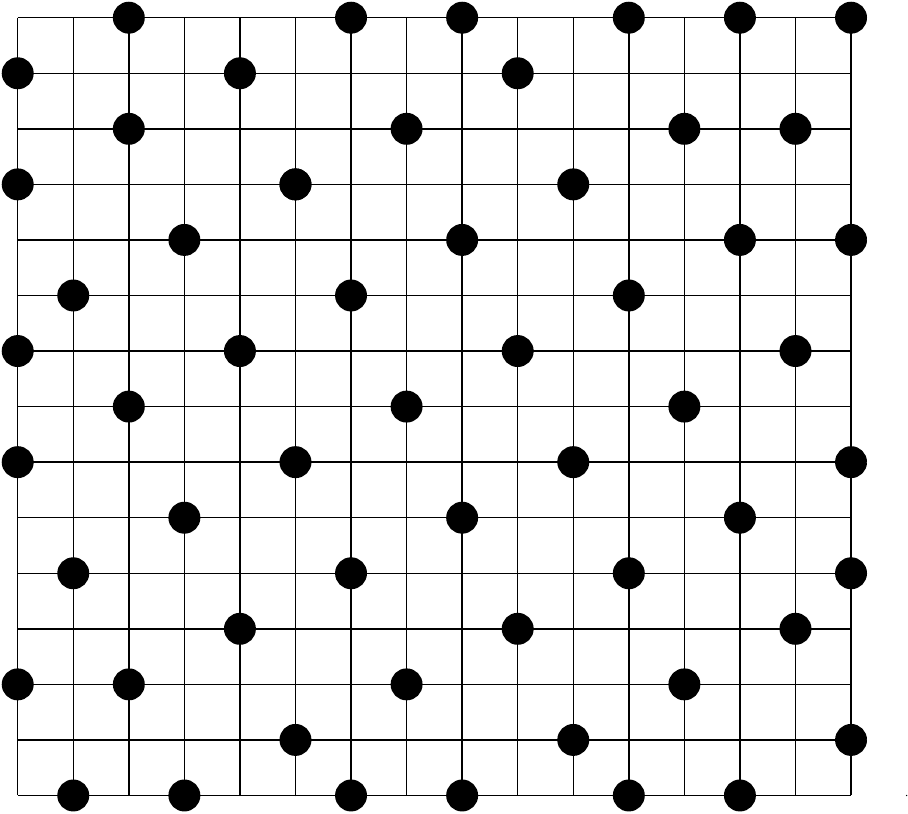}\label{fig:case_15_16}}\hspace{0.4cm}
     \subfigure[$\imath\dot{}_{[1,2]}(P_{15}\Box P_{17})\leq 60$]{\includegraphics[width=0.31\textwidth]{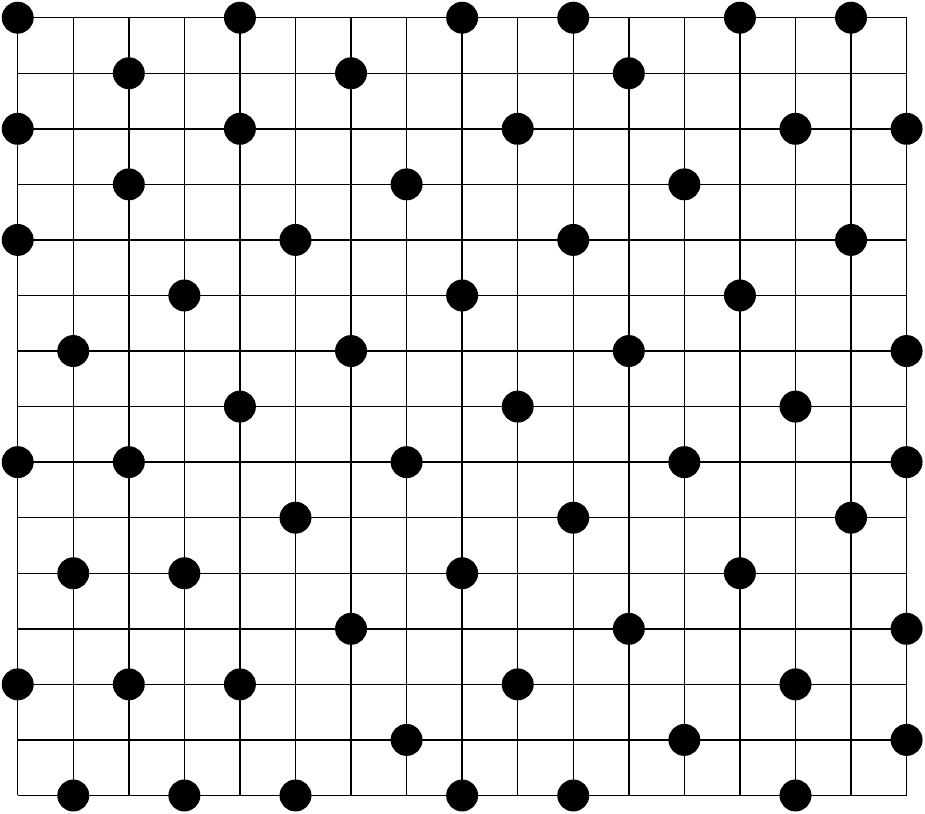}\label{fig:case_15_17}}\hspace{0.4cm}

\caption{Independent $[1,2]$-sets of $P_{15}\Box P_{n}$ with $\Big\lfloor \frac{(15+2)(n+2)}{5}\Big\rfloor-4$ vertices, for $n=15,16,17$.}\label{fig:case_15}
  \end{center}
\end{figure}

For the rest of cases we will use the construction shown in Theorem 3.1 of \cite{CreOs}, that we recall here. Consider the grid with $(m+2)\times (n+2)$ vertices and denote its vertex set $V=\{v_{ij}\colon 0\leq i\leq m+1, 0\leq j\leq n+1\}$. Then the grid $P_m\Box  P_n$ can be identified with the \emph{inner grid}, which is the subgraph induced by the vertex subset $\{v_{ij}\colon 1\leq i\leq m, 1\leq j\leq n\}$. Define a partition of the set $V$ with five subsets $V_s=\{ v_{ij}\colon 2i+j\equiv s\ (mod\ 5)\}$, $s=0,1,2,3,4$. Consider the set $V'_s$ obtained from $V_s$ by replacing each vertex outside the inner grid by its neighbor inside. It is clear that every $V_s$ is a dominating set of the inner grid, so $V'_s$ is a dominating set (not necessarily independent nor $[1,2]$-dominating) of $P_m\Box  P_n$. On the other hand, using that $V_0,V_1,V_2,V_3,V_4$ form a partition of $V$, there exists $s$ such that $|V'_s|\leq |V_s|\leq \big \lfloor \frac{(m+2)(n+2)}{5}\big\rfloor$. Now each $V'_s$ can be modified, deleting some vertices and adding other ones, to obtain $W_s$ an independent $[1,2]$-set of $P_m\Box P_n$ with $|W_s|=|V_s|-4\leq \big\lfloor \frac{(m+2)(n+2)}{5}\big\rfloor - 4$, as desired.

In case $16\leq m\leq n$, it is straightforward to check that the independent dominating set of $P_m\Box  P_n$ shown in Theorem 3.1 of \cite{CreOs} is also a $[1,2]$-set and it has cardinal at most $\big\lfloor\frac{(m+2)(n+2)}{5}\big\rfloor-4$.

Consider now the grid $P_{14}\Box P_n$ with $n\geq 18$. We define the \emph{left strip} as the subgraph generated by vertices $L=\{ v_{ij}\colon 1\leq i\leq 14, 1\leq j\leq 9\}$ and the \emph{right strip} as the subgraph generated by vertices $R=\{v_{ij}\colon 1\leq i\leq 14, n-8\leq j\leq n\}$. Note that both subgraphs have no common vertices and we work separately with them. Firstly let us focus on left strip. For each $s\in \{0,1,2,3,4\}$ we consider the vertex subset $V'_s \cap L$ and we will change some vertices on it, in order to obtain a new set with cardinal $|V_s \cap L|-2$.

For the case $s=0$ the set $V'_0$ is an independent $[1,2]$-set in the left strip with $v_{\stackrel{}{0,0}},v_{\stackrel{}{15,0}}\in V_0$, so $|V'_0 \cap L|=|V_0\cap L|-2$ and no change is needed. For each $s\in \{1,2,3,4\}$ modifications needed to obtain the desired set are shown in Figure~\ref{fig:left_14}.

\begin{figure}[htb]
  \begin{center}
    \subfigure[Case $s=1$]{\includegraphics[width=0.47\textwidth]{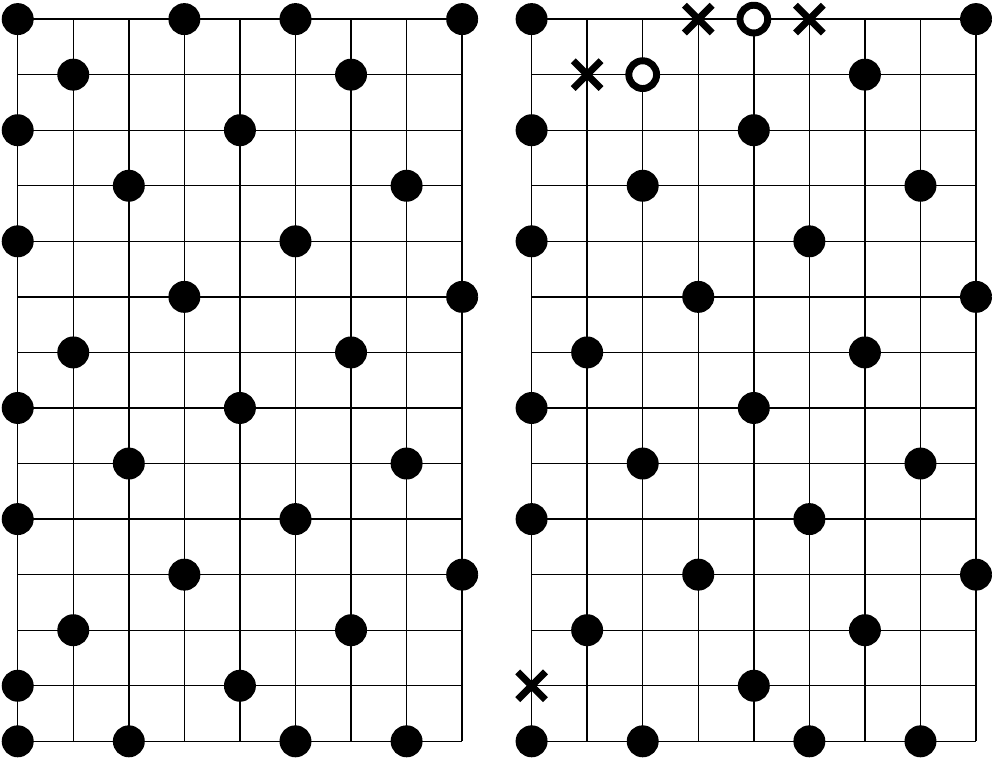}\label{fig:left_14_1}} \hspace{0.5cm}
    \subfigure[Case $s=2$]{\includegraphics[width=0.47\textwidth]{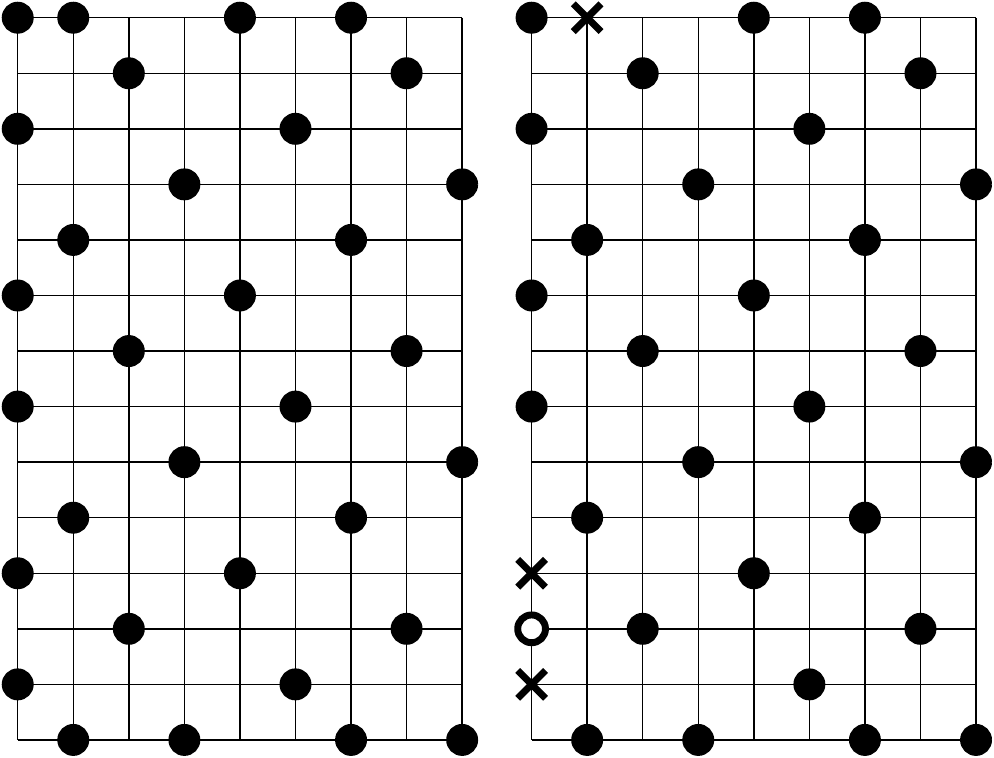}\label{fig:left_14_2}}
     \subfigure[Case $s=3$]{\includegraphics[width=0.47\textwidth]{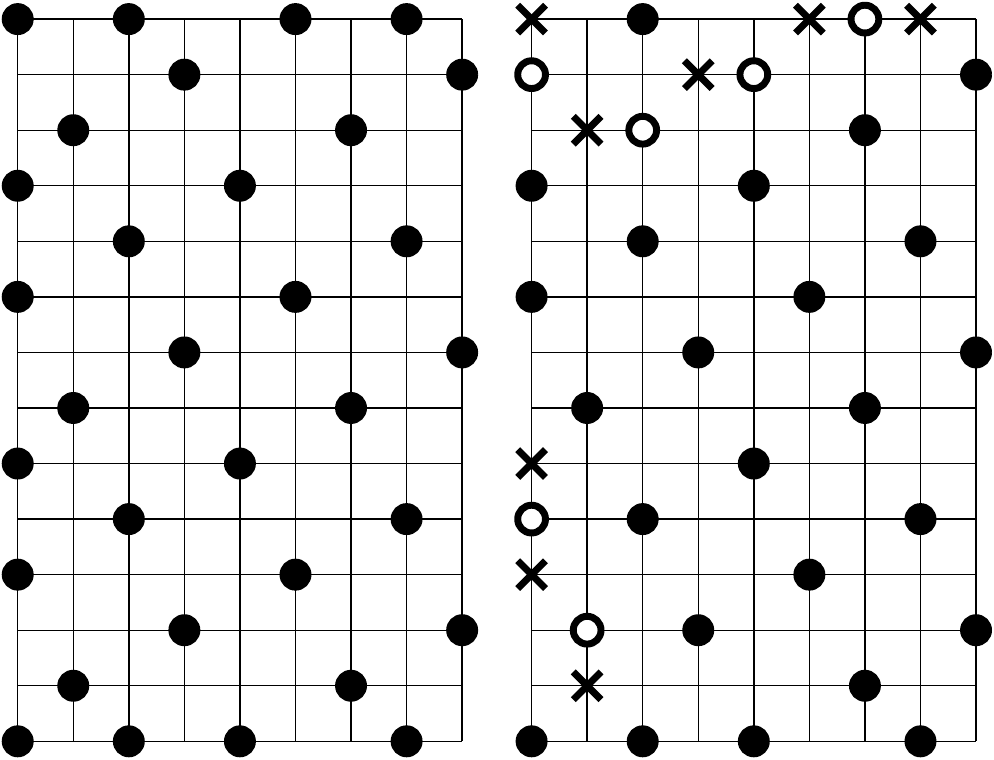}\label{fig:left_14_3}} \hspace{0.5cm}
    \subfigure[Case $s=4$]{\includegraphics[width=0.47\textwidth]{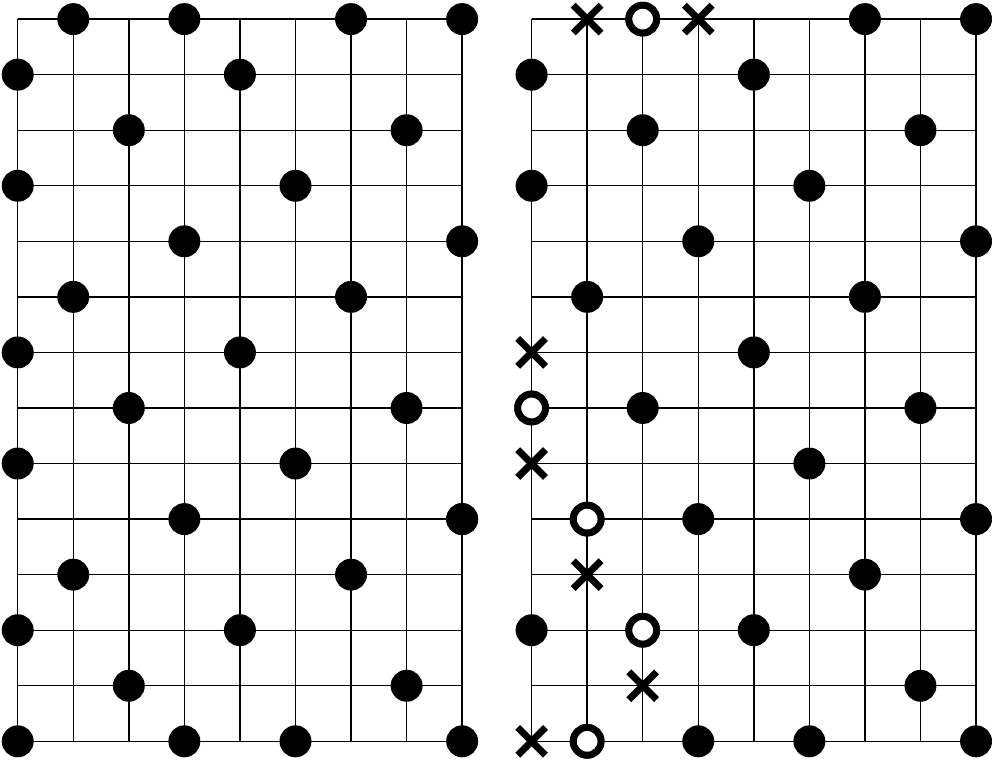}\label{fig:left_14_4}}
\caption{Changes on the left strip of grid $P_{14}\Box P_n$, cases s=1,2,3,4}\label{fig:left_14}
  \end{center}
\end{figure}

In each case, figure on the left shows black vertices in $V'_s\cap L$ and figure on the right shows vertices deleted (crossed) and added (white) to this set. Note that in each case the resulting set is independent, has cardinal two less than the cardinal of the original one, and it $[1,2]$-dominates the first eight columns on the left strip. We would also like to point out that domination of vertices on column number nine and the following ones is not affected by the changes we made.

By symmetry, a similar argument can be used for the right strip. Finally by choosing the appropriate $s\in\{0,1,2,3,4\}$ we obtain the desired independent $[1,2]$-set with cardinal at most $\lfloor \frac{(m+2)(n+2)}{5}\rfloor -4$,  consisting on vertices in both modified strips in addition with vertices of $V'_s$ laying in the inner columns.

The construction for the grid $P_{15}\Box P_n$ with $n\geq 18$ is similar. Firstly for the case $s=0$ the set $V'_0$ is an independent $[1,2]$-set in the left strip with $v_{\stackrel{}{0,0}}\in V_0$, so $|V'_0 \cap L|=|V_0\cap L|-1$ and we need to reduce its cardinal just by one, as it is shown in Figure~\ref{fig:left_15_2}. Similarly if $s=2$ then $V'_2$ is an independent $[1,2]$-set in the left strip with $v_{\stackrel{}{0,16}}\in V_2$, so $|V'_2 \cap L|=|V_2\cap L|-1$ and we need to reduce its cardinal just by one, as it is shown in Figure~\ref{fig:left_15_2}.

\par\medskip

\begin{figure}[htb]
  \begin{center}
    \subfigure[Case $s=0$]{\includegraphics[width=0.47\textwidth]{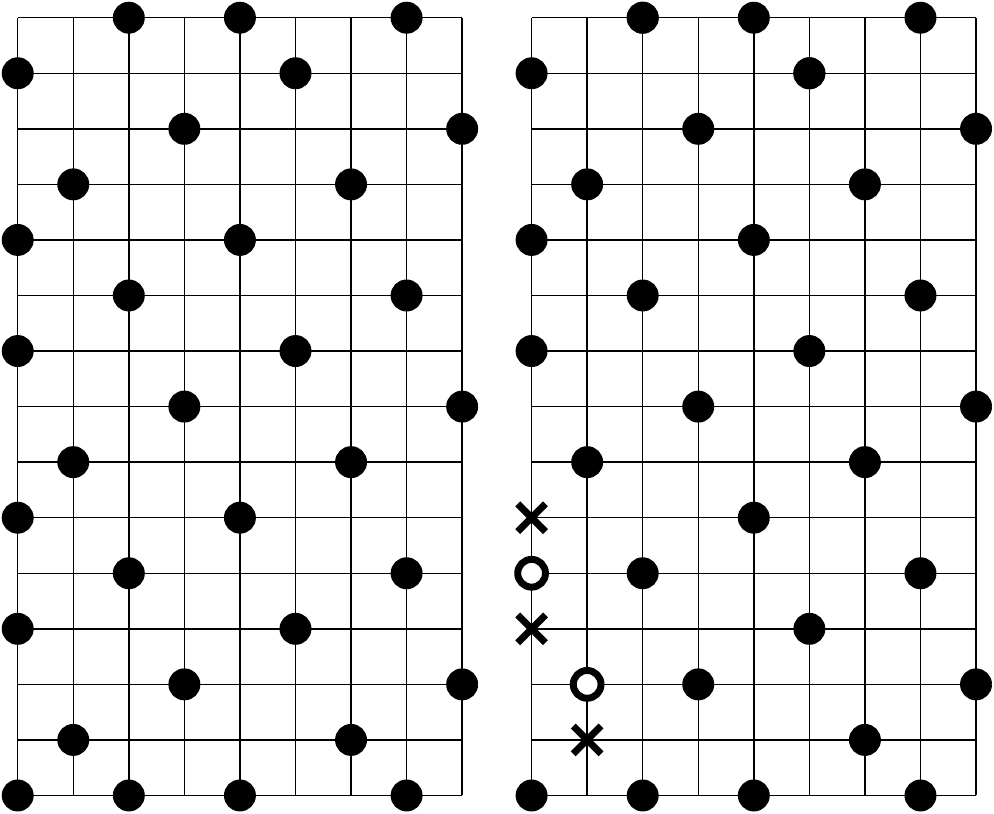}\label{fig:left_15_0}}\hspace{0.5cm}
    \subfigure[Case $s=2$]{\includegraphics[width=0.47\textwidth]{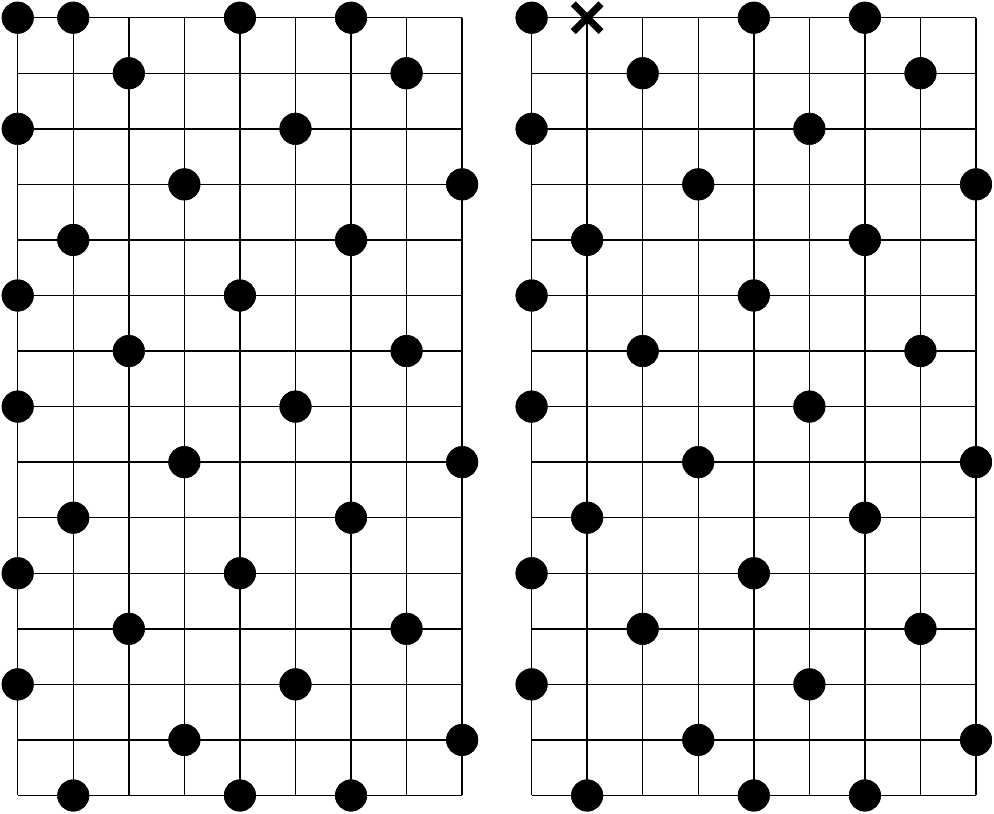}\label{fig:left_15_2}}
    \caption{Changes on the left strip of grid $P_{15}\Box P_n$, cases $s=0$ and $s=2$}\label{fig:left_15}
  \end{center}
\end{figure}

The modifications needed for each $s\in \{1,3,4\}$ are in Figure~\ref{fig:left_15bis} and again changes in the right strip are the similar up to symmetry. Finally, by choosing the appropriate $s\in\{0,1,2,3,4\}$ we obtain the desired independent $[1,2]$-set with cardinal at most $\lfloor \frac{(m+2)(n+2)}{5}\rfloor -4$, consisting on vertices in both modified strips in addition with vertices of $V'_s$ laying in the inner columns.
\qed

\begin{rmk}
Note that our construction of independent $[1,2]$-sets for $P_{14}\Box P_n$ and $P_{15}\Box P_n$ is a modification of the construction shown in Theorem 3.1 of \cite{CreOs}, however we can not apply it directly because the proof in that theorem requires to divide the grid into four vertex-disjoint corners of size $8\times 8$, and this does not hold for these two special cases of grids because of their size.
\end{rmk}

\renewcommand{\textfraction}{0.01}
\begin{figure}[htb]
  \begin{center}
    \subfigure[Case $s=1$]{\includegraphics[width=0.47\textwidth]{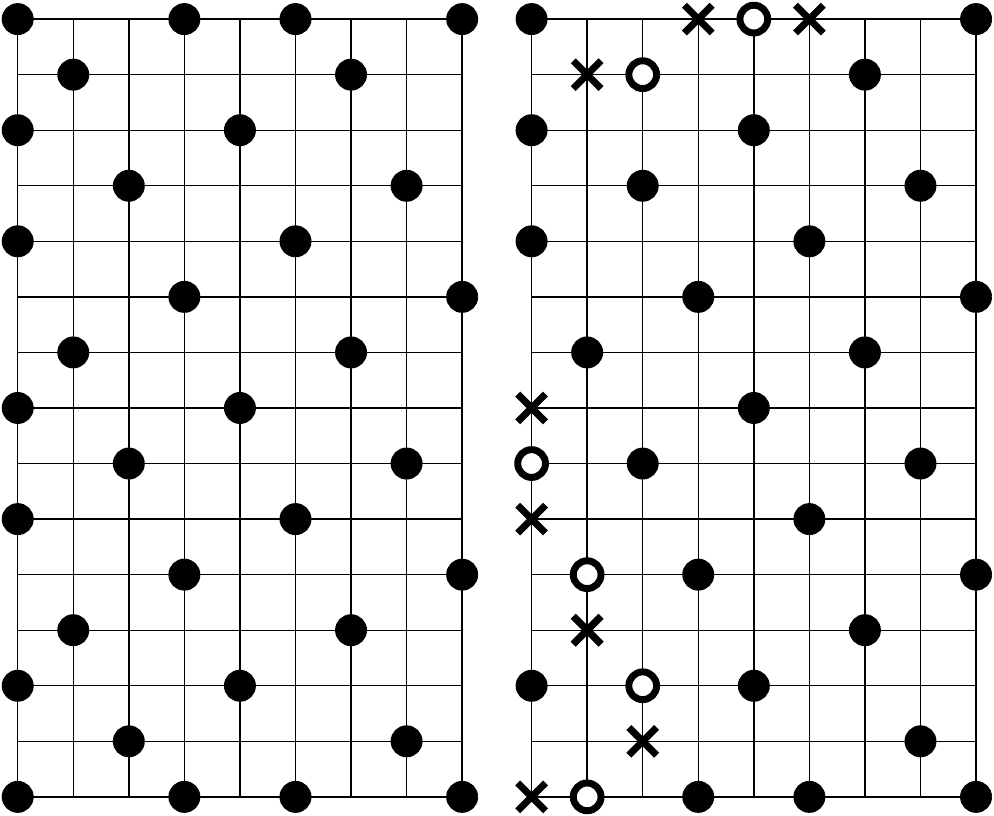}\label{fig:left_15_1}}\hspace{0.5cm}
    \subfigure[Case $s=3$]{\includegraphics[width=0.47\textwidth]{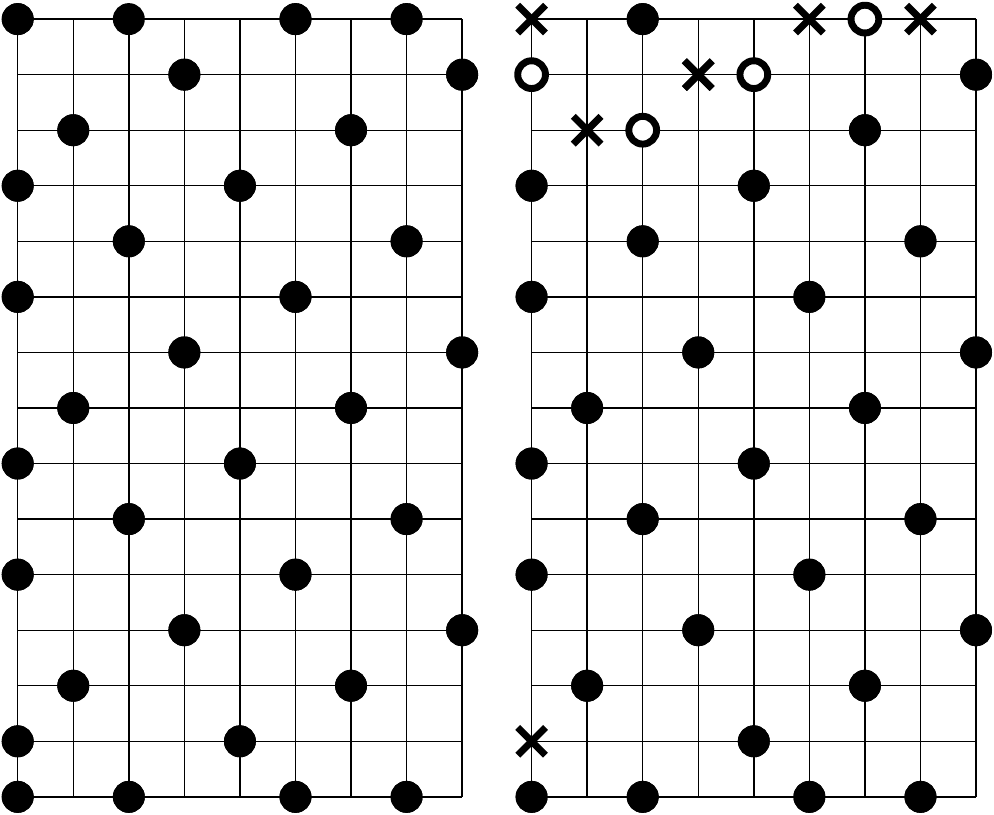}\label{fig:left_15_3}}
    \subfigure[Case $s=4$]{\includegraphics[width=0.47\textwidth]{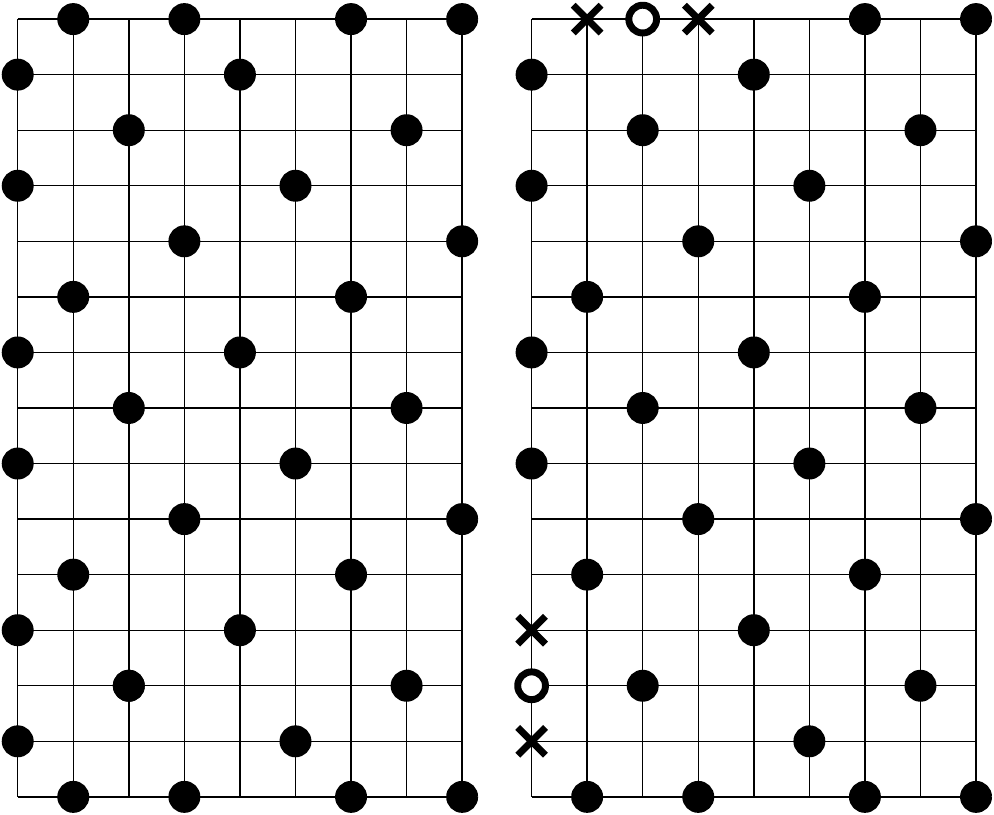}\label{fig:left_15_4}}
\caption{Changes on the left strip of grid $P_{15}\Box P_n$, cases s=1,3,4}\label{fig:left_15bis}
  \end{center}
\end{figure}

\section{Conclusions}

Independent $[1,2]$-sets are a generalization of efficient dominating sets, also called perfect codes, so they can be considered as quasi-efficient dominating sets. In this paper we solve a question posed in \cite{Che} about the existence of these sets in grids and we also obtain the value of the associated parameter $\imath\dot{}_{[1,2]}(P_m\square P_n)$. This question is relevant because the grid $P_m\Box  P_n$ has an efficient dominating set if and only if $m=n=4$ or $m=2, n=2k+1$ \cite{Mar}, so less demanding properties are needed to construct independent dominating sets in grids such that each vertex that is not in such a set has the smallest number of neighbors.

On the one hand, the algorithm presented in Section~\ref{sec:small} gives that the independent $[1,2]$-number agrees with the independent domination number in almost every grid of small size $1\leq m\leq 13, m\leq n$
$$\imath\dot{}_{[1,2]}(P_m\square P_n)=\left\{
\begin{array}{lll}
\ i(P_m\square P_n)+1& \dots &m=12, n\equiv10(mod\ 13) \\
&&\\
\ i(P_m\square P_n)& \dots &\text{otherwise}\\
\end{array}
\right.
$$

On the other hand, in Section~\ref{sec:big} we use a different approach for bigger grids. Although applying the algorithm for large values of $m$ and $n$ theoretically is possible, the running time needed is extremely long. Thus the calculation of $\imath\dot{}_{[1,2]}(P_m\square P_n)$ for $14\leq m\leq n$ is solved using a constructive method, by means of a quasi-regular pattern, that provides the following results

$$\imath\dot{}_{[1,2]}(P_m\square P_n)=\left\{
\begin{array}{lll}
\ i(P_m\square P_n)&\dots &m=14,15,\ m\leq n \\
&&\\
\gamma(P_m\square P_n)&\dots & 16\leq m \leq n
\end{array}
\right.
$$
\par\medskip

The following result is well-known (see \cite{Ban}) and shows the role that efficient dominating sets play into the environment of dominating sets.

\begin{thm}
If $C$ is an efficient dominating set in a graph $G$, then $\vert C\vert = \gamma(G)$.
\end{thm}

This means that, if efficient dominating sets exist in a graph, they are the best dominating sets from two aspects: cardinality, as it is the minimum over all other types of dominating sets, and the number of neighbors each vertex has, which is zero if it is inside the set, and one and only one if it is outside.

Efficient dominating sets are not available in most grids but independent $[1,2]$-sets always exist in $P_m\square P_n$ and they play a similar role regarding to independent domination, except in case $m=12, n\equiv10 \pmod {13}$. In other words, independent $[1,2]$-sets are the best independent dominating sets in grids in both aspects: minimum cardinality and smallest number of neighbors.


\section*{References}

\begin{thebibliography}{99}

\bibitem{Ban} D.W. Bange, A.E. Barkauskas , P.J. Slater, Efficient dominating sets in graphs. Applications of discrete mathematics. Proceedings of the Third SIAM Conference on Discrete Mathematic (1986) 189--199.

\bibitem{Big} N. Biggs, Perfect codes in graphs, J. Combin. Theory Ser. B 15 (1973) 288--296.

\bibitem{Cha} T.Y. Chang, Domination numbers of grid graphs (Ph.D. Thesis), Dept. of Mathematics, University of South Florida, 1992.

\bibitem{chlepi11} G. Chartrand, L. Lesniak, P. Zhang, Graphs and Digraphs, (5th edition). CRC Press, Boca Raton, Florida, 2011.

\bibitem{Che} M. Chellali, O. Favaron, T. Hayes, S. Hedetniemi, A. McRae, Independent $[1,k]$-sets in graphs, Australas. J. Combin. 59 (1) (2014) 144--156.

\bibitem{CreOs} S. Crevals, P.R.J. \"{O}sterg{\aa}rd, Independent domination of grids, Discrete Math. 338 (2015) 1379--1384.

\bibitem{Gon} D. Gon\c{c}alves, A. Pinlou, M. Rao, S. Thomass\'{e}, The domination number of grids, SIAM J. Discrete Math. 25 (3) (2011) 1443--1453.

\bibitem{Gui} D.R. Guichard, A lower bound for the domination number of complete grid graphs, J. Combin. Math. Combin. Comput. 49 (2004) 215--220.

\bibitem{Hay}T.W. Haynes, S.T. Hedetniemi, P.J. Slater (Eds.), Foundamentals of Domination in Graphs, Marcel Dekker, Inc., New York, 1998.

\bibitem{Jac}M.S. Jacobson, L.F. Kinch, On the Domination Number of Products of Graphs: I , Ars Combin. 18 (1984) 33--44.

\bibitem{Mar} M. Livingston, Q.F. Stout, Perfect Dominating Sets. Proceedings of the Twenty-first Southeastern Conference on Combinatorics, Graph Theory, and Computing (Boca Raton, 1990). Congr. Numer. 79 (1990) 187--203.

\bibitem{Pin} J.-E. Pin, Tropical semirings, Idempotency (Bristol, 1994), 50–69, Publ. Newton Inst., Vol. 11, Cambridge Univ. Press, Cambridge, 1998.

\bibitem{Spa} A. Spalding, Min-Plus Algebra and Graph Domination, (Ph.D. thesis), Dept. of Applied Mathematics, University of Colorado, 1998.

\bibitem{Sim}I. Simon, Recognizable sets with multiplicities in the tropical semiring, Mathematical Foundations of Computer Science (Carlsbad, 1988), Lecture Notes in Computer Science, 324 (1988) 107--120.

\bibitem{Viz} V.G. Vizing, The Cartesian product of graphs, (Russian) Vyčisl. Sistemy 9 (1963) 30--43.

\end{thebibliography}
\end{document}